\documentclass{article}

\usepackage[english]{babel}

\usepackage[a4paper,top=2cm,bottom=2cm,left=3cm,right=3cm,marginparwidth=1.75cm]{geometry}
\usepackage{xparse}

\usepackage[colorlinks=true, allcolors=blue]{hyperref}

\usepackage{amsmath,mathtools,amsthm,a4wide}
\usepackage{amssymb, authblk}
\usepackage[mathscr]{eucal}
\usepackage{stackengine} 
\usepackage{bm,booktabs}
\usepackage{bbm}
\usepackage[shortlabels]{enumitem}
\usepackage[usenames,dvipsnames]{xcolor}
\usepackage{comment}
\usepackage{ulem}

\usepackage{tikz}
\usepackage{pgfplots}

\usepackage[labelformat=simple]{subcaption}
\usepackage{graphicx, caption}

\newtheorem{theorem}{Theorem}[section]
\newtheorem{lemma}[theorem]{Lemma}
\newtheorem{proposition}[theorem]{Proposition}

\newtheorem{assumption}{Assumption}

\newcommand{\prob}{\mathbb{P}}
\newcommand{\Prob}[1]{\prob\left(#1\right)}

\newcommand{\E}[1]{\mathbb{E}\left[ #1 \right]}

\newcommand{\plim}{\ensuremath{\stackrel{\prob}{\longrightarrow}}}

\newcommand{\dd}{{\; \rm d}}

\newcommand{\bv}{{\boldsymbol{v}}}

\title{Large deviations for subgraphs in inhomogeneous random graphs}
\author{Riccardo Michielan, Clara Stegehuis, Bert Zwart}
\date{\today}

\begin{document}

\maketitle

\begin{abstract}
Inhomogeneous random graphs are fundamental models for real-world networks, where prescribed degrees are imposed as soft constraints. A common assumption in such models is that the degree distribution follows a power-law, capturing the heavy-tailed nature observed in many contexts. While various graph functionals have been studied in this setting, inhomogeneity makes their analysis significantly more challenging. The goal of this paper is to investigate the large deviations of subgraph counts in inhomogeneous random graphs. Rare events concerning these functionals translate into quantifying the probability that extremely large hubs appear in the graph. This can be achieved by defining a specific optimization problem that captures the most likely way to generate numerous additional subgraphs. When the expected number of subgraphs is sublinear in the graph size, polynomially large deviations are possible, and in this case, we can derive sharp results on clique counts.
\end{abstract}

\section{Introduction}
Many real-world networks exhibit heavy-tailed degree distributions that can be modeled by power-laws, often with infinite second moment~\cite{vazquez2002}. To model such networks, random graph models with heavy-tailed degree distributions have become central benchmarks. These models are designed to replicate the degree of heterogeneity observed in empirical networks, but typically do not enforce other structural properties. As a result, understanding the behavior of various network statistics in such models has become a key area of research~\cite{dhara2019,friedrich2015a,heydari2017,janssen2019,hofstad2017,yin2019a}.
In this paper, we focus on the occurrence of cliques and, more generally, arbitrary fixed subgraphs. The count of such subgraphs provides insight into the network’s local density and so-called network motifs, which have been widely studied in various applications of network science~\cite{milo2002,alon2007,picard2008}. While many real-world networks contain many dense subgraphs, classical sparse random graph models tend to be locally tree-like and thus contain few such structures in the large-network limit. However, in power-law random graphs, the presence of high-degree vertices (hubs) can significantly increase the likelihood of observing dense subgraphs, even when the average degree remains finite~\cite{hofstad2017b,sonmez2024distances}. This raises a natural question: How unlikely is it for a sparse power-law random graph to contain an unusually large number of cliques or other fixed subgraphs?

Large deviation principles for subgraph counts have been extensively studied in Erdős–Rényi graphs. For example, the tail behavior of triangle counts has been characterized in both dense and sparse regimes~\cite{janson2004,janson2004a,kim2004,demarco2011,chakrabarty2021,chatterjee2012}, revealing that rare events with many triangles are typically driven by the emergence of localized dense regions. Other work studies the large deviation behavior of random walks and other stochastic processes on graphs ~\cite{Coghi2019, Altarelli2013}. These insights have been extended to general subgraphs in the dense setting, but the sparse case—especially under heavy-tailed degree distributions—remains far less understood.

Existing large deviation results for inhomogeneous random graphs often rely on the assumption of dense graphs~\cite{Markering2020LDP} or finite second moments~\cite{andreis2021,chakrabarty2021,dommers2018}, which allow for tractable analysis via factorized connection probabilities. In contrast, when the degree distribution has infinite variance, the connection probabilities exhibit nonlinear dependencies on vertex weights, introducing degree-degree correlations that complicate the analysis~\cite{stegehuis2017b,yao2017}. Some progress has been made in regimes where the average degree diverges~\cite{oliveira2019}, but these results do not apply to the sparse, power-law regime. To date, large deviation results in this setting are limited to specific functionals such as PageRank~\cite{Mariana2021}, edge counts~\cite{kerriou2022,stegehuis2022scale}, giant component \cite{jorritsma2024largedeviationsgiantcomponent}, and triangle counts~\cite{stegehuis2023}.

In this paper, we derive asymptotic estimates for the upper tail probabilities of clique counts and a class of general subgraph counts in sparse, inhomogeneous random graphs with power-law degree distributions that have finite mean and infinite second moment. For $k$-cliques, we identify the precise decay rate of the tail probability and show that the dominant contribution to the deviation arises from the presence of $k-2$ hubs with degrees significantly exceeding the typical maximum degree. For general subgraphs, these most likely structures may look different, and we characterize the subgraph structures and degree configurations that drive large deviations. This highlights the interplay between subgraph topology and the heavy-tailed nature of the degree distribution.

\paragraph{Notation.}
We denote $[k]=\{1,2,\dots,k\}$. We say that a sequence of events $(\mathcal{E}_n)_{n\geq 1}$ happens with high probability (w.h.p.) if $\lim_{n\to\infty}\Prob{\mathcal{E}_n}=1$ and we use $\plim $ for convergence in probability. We write $f(n)=o(g(n))$ if $\lim_{n\to\infty}f(n)/g(n)=0$, and $f(n)=O(g(n))$ if $|f(n)|/g(n)$ is uniformly bounded. We write $f(n)=\Theta(g(n))$ if $f(n)=O(g(n) )$ as well as $g(n)=O(f(n))$. We say that $X_n=O_{{\prob}}(g(n))$ for a sequence of random variables $(X_n)_{n\geq 1}$ if $|X_n|/g(n)$ is a tight sequence of random variables, and $X_n=o_{{\prob}}(g(n))$ if $X_n/g(n)\plim 0$.

\paragraph{Organization of the paper.} We first describe the version of the inhomogeneous random graph that we study in Section~\ref{sec:results}, and proceed to state our main results on clique counts and more general subgraph counts. We then provide a discussion in Section~\ref{sec:discussion}, and the proofs of our results for cliques are in Sections~\ref{sec:proofcliques}, and for general subgraphs in Section~\ref{sec:proofsubgraphs}. 

\section{Main results}\label{sec:results}
\subsection{Model}
We consider the rank-1 inhomogeneous random graph (IRG)~\cite{chung2002} with heavy-tailed weights. The IRG is constructed by assigning a positive weight $W_i$ to each vertex $i \in [n]$. These weights are assumed to be independent and identically distributed (i.i.d.) samples from a random variable $W$ with values on the interval $[1,\infty)$ and with probability density function $f_W(x) = \alpha x^{-\alpha-1}(1+o(1))$, where $\alpha \in (1,2)$. In particular, $W$ is heavy-tailed, and it satisfies
\begin{equation}\label{eq:Fx}
    \bar{F}(x) := \mathbb{P}(W_i > x) = x^{-\alpha}(1+o(1)) , \quad x \geq 1.
\end{equation}
Edges between vertices $i$ and $j$ are formed independently with probability
\begin{equation}
    p_{ij} = \Theta\left(\min\left(\frac{W_i W_j}{\mu n}, 1 \right)\right),
\end{equation}
where $\mu = \mathbb{E}[W_i]$ is the expected vertex weight. This connection rule ensures that the expected degree of vertex $i$ scales proportionally to its weight, i.e., $\mathbb{E}[\deg(i)] \approx W_i$~\cite{chung2002,stegehuis2017}, up to constant factors. This kernel includes the Chung-Lu model~\cite{chung2002}, and the generalized random graph~\cite{britton2006generating}.

Then, given a subset of $k$ vertices with weights $(h_1, \dots, h_k)$, the probability that they form a clique is asymptotically proportional to
\begin{equation}
    f_n(h_1, \dots, h_k) := \prod_{1 \leq i < j \leq k} \min\left(\frac{h_i h_j}{\mu n}, 1\right).
\end{equation}

\subsection{Large Deviations for Clique Counts}

The number of $k$-cliques in IRG, denoted by $\mathcal{K}_n^{(k)}$, is a random variable influenced by two sources of randomness: the vertex weight assignment and the random edge formation. Understanding the typical and atypical behavior of $\mathcal{K}_n^{(k)}$ is crucial for characterizing the higher-order connectivity structure of the graph.

We begin by recalling a result from \cite{vdhofstad21,stegehuis2019} that estimates the expected number of $k$-cliques, denoted by $m_n^{(k)} := \mathbb{E}[\mathcal{K}_n^{(k)}]$.

\begin{lemma}\label{lem:averagecliquenumber}
Let $\alpha \in (1,2)$. Then
\[
m_n^{(k)} = H \left(n \bar{F}(\sqrt{n})\right)^k (1 + o(1)),
\]
for some constant $H > 0$. In particular, $m_n^{(k)}$ is regularly varying with index $k(2-\alpha)/2$.
\end{lemma}

This result shows that the expected number of cliques grows polynomially with $n$. Furthermore, it can be shown that when $\alpha\in(1,2)$, most cliques are formed around a small number of vertices with weight approximately $\sqrt{n}$.

The central goal of this section is to establish a large deviation principle for $\mathcal{K}_n^{(k)}$. Specifically, for any $a > 0$, we aim to quantify the probability of observing significantly more cliques than expected,
\[
\mathbb{P}(\mathcal{K}_n^{(k)} > (1+a) m_n^{(k)}).
\]

To analyze this probability, we will study the most likely mechanism for generating an excess number of cliques. We find that the dominant contribution to large deviations in clique counts comes from the presence of $k-2$ exceptionally large hubs, around which these cliques form. The weights of these hubs depend intricately on the desired deviation. Nonetheless, one can understand what is the smallest size required from these $k-2$ hubs to generate an expected excess of $a m_n^{(k)}$ cliques, defining
\begin{equation}\label{eq:smallest_hub_size}
    c_a(n) := \inf\left\{ c : \binom{n}{2} \int_0^\infty \int_0^\infty f_n(x,y,c,...,c) \dd F(x) \dd F(y)  > a m_n^{(k)} \right\}.
\end{equation}

The following lemma characterizes the asymptotic scaling of $c_a(n)$:

\begin{lemma}\label{lem:ca(n)}
Let $c_a(n)$ be as in~\eqref{eq:smallest_hub_size}. 
Then,
\[
c_a(n) = K_2 a^{\frac{1}{2(\alpha-1)}} n^{\alpha^*}(1+o(1)),
\]
where $\alpha^* = 1 - \frac{2 - k(2-\alpha)}{4(\alpha-1)}$, for some positive constant $K_2$.
\end{lemma}

We now present the main result of this section, which provides matching asymptotic bounds for the order of magnitude of the large deviation probability for clique counts.

\begin{theorem}\label{thm:LDcliques}
Let $\alpha > 2 - 2/k$ and $a > 0$. Then, as $n \to \infty$,
\[
\mathbb{P}(\mathcal{K}_n^{(k)} > (1+a) m_n^{(k)}) = \Theta \left(n \mathbb{P}(W > c_a(n))\right)^{k-2}.
\]
\end{theorem}

Theorem~\ref{thm:LDcliques} shows that the probability of observing a large deviation in the number of $k$-cliques is asymptotically equivalent in order of magnitude to the probability of having $k-2$ hubs of size $c_a(n)$. This highlights the role of extreme vertex weights in the tail behavior of $\mathcal{K}_n^{(k)}$. The theorem only holds for $\alpha$ sufficiently large, larger than $2-2/k$. We believe that below this, large deviations are caused by a polynomially growing number of hubs, leading to exponentially decaying deviation probabilities, as our optimization-based approach shows that a finite number of hubs will, in expectation, not yield the correct number of $k$-cliques. Indeed, when $\alpha=2-2/k$, then $\alpha^*=1$. This means that the weight of the added hub(s) is linear. Increasing the weight further, however, does not increase the degree of the node anymore, because of the minimum in the connection probability~\eqref{eq:Fx}. Thus, from this point onward, increasing the hub weight does not add any more copies of $K_k$, and more hubs are necessary to create additional copies.

When choosing $a = a(n)$ growing with $n$, we can still apply Theorem \ref{thm:LDcliques} to compute the probability of larger deviations, such as polynomial deviations. In particular, consider deviations of order $n^\gamma \gg m_n^{(k)}$, for $\gamma > 0$ sufficiently large. Lemma \ref{lem:averagecliquenumber} yields $n^\gamma = a m_n^{(k)}(1+o(1))$, with $a = n^{\gamma - k(2-\alpha)/2}(1+o(1))$. Then, Lemma \ref{lem:ca(n)} gives
$$c_a(n) = \Theta\left(a^{\frac{1}{2(\alpha-1)}}n^{1-\frac{2 - k(2-\alpha)}{4(\alpha -1)}}\right) = \Theta\left(n^{1 + \frac{\gamma -1}{2(\alpha-1)}}\right),$$
from which we can obtain the probability of deviations using Theorem \ref{thm:LDcliques}. Interestingly, the scale at which $c_a(n)$ grows in $n$ is independent of $k$, meaning that the required size of hubs necessary to achieve prescribed polynomial clique deviation is independent of the clique size.

\subsection{Deviations for subgraph counts}
Next, we investigate deviations of other subgraph counts than cliques. 
Given a small graph $H$ on $k$ vertices, we denote by $N(H)$ the number of times that $H$ occurs as a subgraph in the IRG. 
Furthermore, we denote the conditional expectation of $N(H)$ given the weights of all $n$ vertices $W_1, W_2,..., W_n$,
\begin{equation}
    C^{(n)}(H) = \mathbb{E}\left[N(H) \;|\; \{W_i\}_{i \in [n]}\right].
\end{equation}
Our goal is to calculate $$ \Prob{C^{(n)}(H) > n^{\gamma}},$$ for some $\gamma >0$. 

We will focus on a specific class of subgraphs:
\begin{assumption}\label{ass:feasible_solution}
Consider the optimization problem 
\begin{align}\label{eq:opt_problem_expectation}
    B(H) := \max_{\boldsymbol{\beta} \in [0,1]^k} \quad  &  \sum_{i \in [k]} (1 -  \beta_i) + \sum_{(i,j) \in H} \min(\beta_i + \beta_j - 1, 0).
\end{align}
The optimal solution of \eqref{eq:opt_problem_expectation} for graph $H$ satisfies $\boldsymbol{\beta}^*\in [0,1/\alpha)^k$. 
\end{assumption}
Informally, the objective function in Equation \eqref{eq:opt_problem_expectation} encodes the polynomial scaling (with respect to $n$) of the expected number of copies of $H$ formed on vertices with weights proportional to $(\beta_1,...,\beta_k)$; thus, $B(H)$ identifies the expected subgraph count in the IRG, and the optimal $\boldsymbol{\beta}$ suggest which weight configuration yields maximum contribution. 
{Furthermore, by~\cite[Lemma 4.2]{vdhofstad21}, this assumption implies that the optimal solution is unique.} This set of graphs is quite large, and contains, for example, all cliques and Hamiltonian graphs with an odd number of nodes.

Assumption \ref{ass:feasible_solution} ensures that the random variable $N(H)$ concentrates around its mean. If Assumption~\ref{ass:feasible_solution} does not hold, then, with high probability, $N(H)$ has a smaller order of magnitude than its expectation, and therefore the probability of deviation from its mean is considerably large. Furthermore Assumption~\ref{ass:feasible_solution} ensures that
\begin{equation}\label{eq:convergencesubgraphs}
    \frac{N(H)}{n^{k(2-\alpha)/2}}\plim K_H
\end{equation}
for some $0<K_H<\infty$~\cite[Theorem 2.2]{vdhofstad21}. We now define
\begin{align}
R(H):=\max_{\boldsymbol{\beta}\in[0,1]^k} & \sum_{i \in [k]} \min \left(1 - \alpha\beta_{i},0\right) \label{eq:opt_poly_obj} \\
\text { s.t. } & \sum_{i \in V_{H}} \max \left(1 - \alpha\beta_{i},0\right)+\sum_{\{i, j\} \in E_{H}} \min \left(\beta_{i}+\beta_{j}-1,0\right) \geq \gamma .\label{eq:opt_poly_constr}
\end{align}
The following theorem shows that $R(H)$ characterizes the rate of the deviations of $N(H)$:
\begin{theorem}\label{thm:subgraphdeviation}
    Let $\gamma>k(2-\alpha)/2$,  and let $H$ satisfy Assumption~\ref{ass:feasible_solution}. When~\eqref{eq:opt_poly_constr} is feasible, then
    \begin{equation}
        R(H)= \lim_{n\to\infty}\frac{\log(\Prob{C^{(n)}(H) > n^{\gamma}})}{\log(n)}.
    \end{equation}
\end{theorem}
Under Assumption \ref{ass:feasible_solution}, if $\gamma \leq B(H)$ the solution of \eqref{eq:opt_poly_obj}+\eqref{eq:opt_poly_constr} is trivially $R(H) = 0$. In other words, $n^\gamma$ is below the typical value, and there is no need to introduce extra-large hubs to observe $n^\gamma$ copies of $H$ in IRG. 
The problem \eqref{eq:opt_poly_obj}+\eqref{eq:opt_poly_constr} is piece-wise linear; however, it can be rewritten as a Mixed-Integer linear program (MILP):

\begin{align}
\max & \sum_{i \in V_{H}} \delta_{i}  \\
\text { s.t. } & \sum_{i \in V_{H}} \theta_{i}+\sum_{i, j \in E_{H}} \zeta_{i j}>\gamma \\
& \gamma_{i} \leq -\beta_{i}\alpha+1, \quad \delta_{i} \leq 0 \\
& \zeta_{i j} \leq \beta_{i}+\beta_{j}-1, \quad \zeta_{i j} \leq 0 \\
& \theta_{i} \leq 1-\beta_{i}\alpha+1-b_{i} \\
& \theta_{i} \leq b_{i} \\
& b_{i} \in\{0,1\}
\end{align}
Indeed, now the term $\min(1-\alpha\beta_i,0)$ is encoded by $\gamma_{i}$. The constraints $\gamma_i\leq 0$ and $\gamma_i\leq 1-\alpha\beta_i$ combined with the fact that $R(H)$ is a maximization problem, ensure that indeed $\gamma_i = \min(1-\alpha\beta_i,0)$. In a similar manner, $\zeta_{i,j}$ encodes $\min(\beta_i+\beta_j-1,0)$. The terms $\max(1-\alpha\beta_i,0)$ are slightly more complicated, and will be represented by $\theta_i$. To achieve this, we introduce extra binary variables $b_i$. When $b_i=0$, then the constraints on $\theta_i$ read $\theta_i\leq 0$ and $\theta_i\leq 2-\alpha\beta_i$. The latter constraint is always satisfied, as $\alpha\in(1,2)$ and $\beta_i\in[0,1]$. On the other hand, when $b_i=1$, then the constraints on $\theta_i$ read $\theta_i\leq 1$ and $\theta_i\leq 1-\alpha\beta_i$. In this case, the first constraint is always satisfied. Thus, the maximization problem will set $b_i=1$ whenever $1-\alpha\beta_i>0$, and zero otherwise. 

Intuitively, the optimal value of $\beta_i$ gives the scaling of the minimal weight necessary to have $n^\gamma$ copies of $H$. That is, a typical copy of $H$ in a graph with at least $n^\gamma$ copies of $H$, the vertex isomorphic to $i$ has weight proportional to $n^{\beta_i}$.

\begin{figure}
    \centering
    \includegraphics[width=\linewidth]{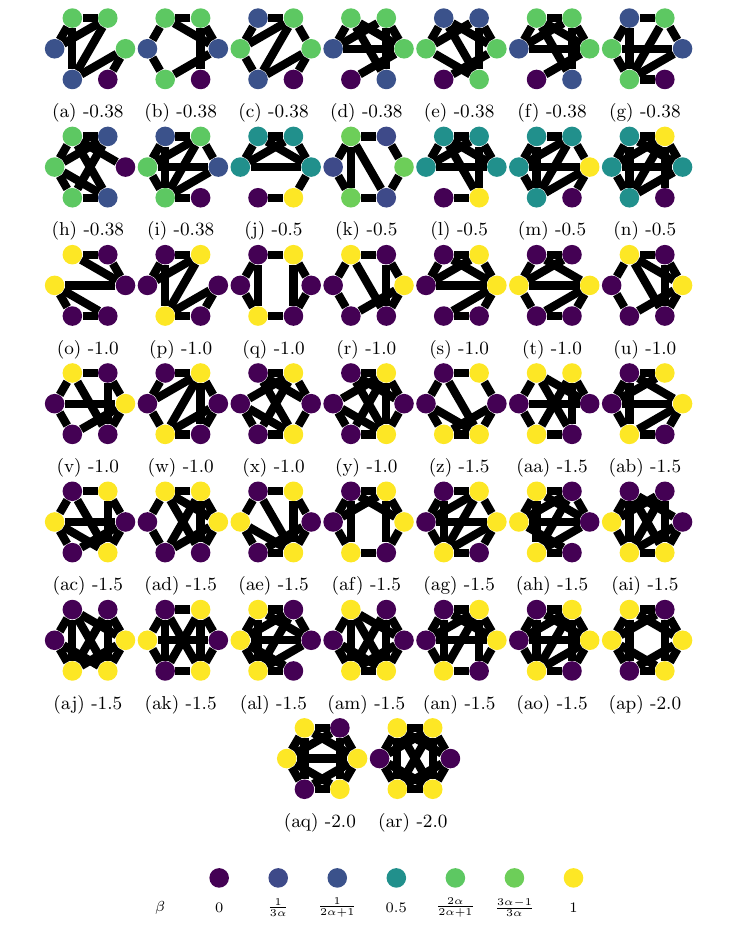}
    \caption{Deviations for all subgraphs of size 6 for $\alpha=1.5$, $\gamma = 2$. The colors indicate the optimal $\beta$ values of~\eqref{eq:opt_poly_constr}, and the subgraph captions demonstrate $R(H)$.}
    \label{fig:ldp_subgraph_tau_25}
\end{figure}

\section{Discussion}\label{sec:discussion}
\paragraph{Unconditional subgraph deviations.}
In Theorem~\ref{thm:subgraphdeviation}, we develop the large deviations scaling of the number of copies of $H$, conditionally on the weights, thus ignoring the randomness from the edge probabilities conditional on the weights. The main problem in developing large deviations principles for the unconditional version is that the occurrences of copies of $H$ are not independent events. Several copies of a subgraph $H$ may overlap, making their appearances correlated. We believe that these sources of randomness do not dominate the large deviation problem, and conjecture that the same scaling holds for the unconditional version. 
{For cliques, we deal with these dependencies in Lemma~\ref{lem:kcliqueshare}, where we use the fact that a $k$-clique always contains the subgraph which is $k-2$ triangles glued together at one edge, in which case the $k-2$ triangles are independently present conditionally on the one edge being present. For general subgraphs, this argument does not hold, and a different approach is necessary.}

\paragraph{Other subgraphs.} 
Theorem~\ref{thm:subgraphdeviation} only holds for a specific class of subgraphs. For large $k$, this class is quite large, as it contains all Hamiltonian subgraphs among the others. We believe that there exists a different class of subgraphs for which a different approach is necessary. Indeed, many subgraph counts are known not to concentrate and to depend on the presence of hubs, so their large deviation properties are expected to differ significantly.

\paragraph{General $U$-statistics.}
Instead of focusing on subgraph counts, it can also be interesting to focus on the large deviations properties of general U-statistics. Given a graph $G$ and an integer $k$, a U-statistic is a function of the kind
$$F_n(G) := \sum_{v_1,...,v_k \in V} h_n(v_1,...,v_k)$$
for some local function $h_n : V^k \to \mathbb{R}$. U-statistics include subgraph counts or homomorphism densities, as well as other relevant global observables such as clustering coefficients, several centrality measures, and average shortest path distances. 
Any two vertices with the same weight are indistinguishable in IRG; thus, with a slight abuse of notation, we can define $h_n$ as a function of the weights. It would be interesting to see the large deviation properties of these functions. We believe that our methods here could extend to specific types of U-statistics, for which 
\begin{equation}
	g(\beta_1,\dots,\beta_k):=\frac{\lim_{n\to\infty}\log(h_n(n^{\beta_1},...,n^{\beta_k}))}{\log(n)} \in (0,\infty),
\end{equation}
for all $\boldsymbol{\beta}\in[0,1]^k$, and
\begin{equation}
\lim_{n\to\infty} \frac{\log(h_n(n^{\beta_1},\dots n^{\beta_k}))}{\log(h_n(b_1n^{\beta_1},\dots b_kn^{\beta_k}))} = 1,
\end{equation}
for all fixed $b_1,\dots, b_k>0$. Finally, we assume that
\begin{align}
	B_F := \max_{\boldsymbol{\beta}} \quad  &  \sum_{i = 1}^k (1 - \alpha\beta_i) + g(\beta_1,...,\beta_k)
\end{align}
has feasible solution $\boldsymbol\beta^* \in [0, 1/\alpha]^k$, which is a necessary condition for $F_n$ to concentrate around its mean. In that case,~\eqref{eq:opt_poly_obj}+\eqref{eq:opt_poly_constr} becomes
\begin{align}
	R_F:=\max_{\boldsymbol\beta\in[0,1]^k} & \sum_{i \in V_{H}} \min \left(1 - \alpha\beta_{i},0\right) \nonumber\\
	\text { s.t. } & \sum_{i \in V_{H}} \max \left(1 - \alpha\beta_{i},0\right)+g(\beta_1,\dots,\beta_k)>\gamma.
\end{align}
Whether this optimization problem is efficiently solvable depends on the shape of the function $g(\beta_1,\dots,\beta_k)$. 
{For recent work on the case where $h_n$ does not depend on $n$, we refer to \cite{bakhshizadeh2023}.}

\paragraph{Exponential deviations.}
When~\eqref{eq:opt_poly_obj}+\eqref{eq:opt_poly_constr} is infeasible, then it is not possible to obtain the desired deviation using only a finite number of hubs, and a growing number of hubs is necessary to obtain the desired deviation. The probability of having at least $un^{\theta}$ vertices of weight at least $n^\delta$ is 
\begin{align}
    &\Prob{un^{\theta}\text{ vertices of weight }>n^{\delta}} = \exp \left(-u(\theta-1+\alpha\delta)n^{\theta} \log \left(n\right)\right)(1+o(1)),
\end{align}
when $\theta>1-\alpha\delta$. This equation shows that increasing $\theta$ influences the probability to a much larger extent than $\delta$. We therefore first assume that $\delta = 1$. We again write an optimization problem similar to~\eqref{eq:opt_poly_constr}, where we maximize the probability that the hubs are present, subject to the constraint that they form the desired number of subgraphs. Now, conditionally on $n^\beta$ vertices of weight proportionally to $n$ to be present, the number of vertices of weight at least $n^\beta_i$ scales as $n^{\min(\beta_i\alpha+1,\theta)}$. Indeed, as at least $n^\theta$ vertices of weight proportionally to $n$ are present, this number is at least $n^\theta$, and is otherwise dictated by the power-law distribution. The probability that a subgraph forms is similar to that in~\eqref{eq:opt_poly_constr}. This yields the optimization problem
\begin{align}
	& \min_{\mathbf{\alpha}\in[0,1]^k,\theta\geq 0} \theta \nonumber\\
	& \text { s.t. } \sum_{i \in V_{H}} \max \left(\beta_{i}\alpha+1, \beta\right) +\sum_{i, j \in E_{H}} \min \left(\beta_{i}+\beta_{j}-1,0\right)>\gamma
\end{align}
As~\eqref{eq:opt_poly_constr}, this optimization problem is infeasible when the desired deviation is not possible (for example, more than $n^{3}$ triangles). Again, this can be rewritten as a Mixed Integer Linear Program. When polynomial deviations suffice, setting $\beta=0$ gives a feasible solution, as the constraint then reduces to the same one as in~\eqref{eq:opt_poly_constr}. 

\section{LD for conditional number of cliques}
An important ingredient in proving Theorem~\ref{thm:LDcliques} is showing that the conditional expectation of $\mathcal{K}_n^{(k)}$ given the weights of all $n$ vertices $W_1, W_2,..., W_n$,
\begin{equation}
    C_n = \mathbb{E}\left[\mathcal{K}_n^{(k)} \;|\; \{W_i\}_{i \in [n]}\right],
\end{equation}
concentrates around its mean. This, combined with the following proposition on the large deviations properties of $C_n$, will prove Theorem \ref{thm:LDcliques}.

\begin{proposition}\label{prop:LDconditionalcliques}
    Assume $\alpha > 2 - 2/k$. Then,
    \begin{equation}
        \Prob{C_n > (1+a) m_n^{(k)}} = \Theta \left(n \Prob{W > c_a(n)}\right)^{k-2}.
    \end{equation}
\end{proposition}
The rest of this section will be devoted to proving Proposition~\ref{prop:LDconditionalcliques}.

We now analyze the number of $k$-cliques conditionally on the $n$ weights. 
Let $C_n$ be expected number of $k$-cliques, given weights $\{W_i\}_{i=1,...,n}$. We aim to get upper and lower bounds for $\Prob{C_n > (1+a)m_n^{(k)}}$. We will see that the most likely way to generate additional $a \cdot m_n^{(k)}$ cliques is to have a certain number of hubs in the graph. The conditional number of cliques is given by
\begin{equation}
    C_n = \binom{n}{k} \int_{0}^{\infty} \cdots \int_{0}^{\infty} f_n(x_1,...,x_k) \dd F_n(x_1) \cdots \dd F_n(x_k),
\end{equation}
where $F_n$ is the empirical weight distribution
\begin{equation}
    F_n(x) = \frac{1}{n} \sum_{i}^n \mathbbm{1}_{(W_i \leq x)}.
\end{equation}
We make use of a convenient upper bound for the complementary empirical distribution function $\bar{F}_n$. Define $a(n) = \bar{F}^{-1}(1/n)n^{-\delta}$ and $b(n) = \bar{F}^{-1}(1/n)n^{\delta}$. In particular, $a(n)$ and $b(n) $ are regularly varying of index $1/\alpha - \delta$ and $1/\alpha + \delta$. Assume $\delta \in (0, 1/\alpha -1), c>0, A>0$ and let $E_n(A,c,\delta)$ be the event
\begin{equation}
    \left\{ 
    \sup_{x < a(n)} \frac{\bar{F}_n(x)}{\bar{F}(x)} \leq 1 + A, 
    \sup_{x \in [a(n),b(n)]} \bar{F}_n(x) \leq (1+A) \bar{F}(a(n)), 
    \sup_{x > b(n)} \bar{F}_n \leq (1 + A)c/n
    \right\} \nonumber.
\end{equation}
By Proposition 2.2 in \cite{stegehuis2023},  $\Prob{E_n(A,c,\delta)} \geq 1 - o(n^{-\beta})$ for any fixed $A>0$, as soon as $\lceil c \rceil > \beta/\alpha \delta$. In particular, on the event $E_n(A,c,\delta)$, the empirical distribution function can be bounded by
\begin{equation}\label{eq:empdistrbound}
    \bar{F}_n^*(x) = (1+A) \left[ \bar{F}(x) \mathbbm{1}_{(x < a(n))} + \bar{F}(a(n)) \mathbbm{1}_{(a(n) \leq x < b(n))} + c/n \mathbbm{1}_{(x \geq b(n))} \right].
\end{equation}

\subsection{Contribution from big hubs}
Assume the IRG contains $h$ hubs. A hub is a vertex $i$ with weight $W_i \geq b(n)$, which is regularly varying of index $\alpha_b \in (1/2,1)$. We want to compute the contribution of these $h$ hubs to the total number of cliques in the graph. Formally, let $\mathcal{B} = \{b_1(n),...,b_h(n)\}$ be regularly varying of indices $\alpha_{b_1},...,\alpha_{b_h} \in (1/2,1)$. Assume that the IRG contains $h$ hubs with weights $\mathcal{B}$. Then, the expected number of cliques generated on these hubs is $\binom{n - h}{k - h} S_{\mathcal{B}}^{(h)}(n)$, with
\begin{equation}
    S_{\mathcal{B}}^{(h)}(n) := \int_{1}^{\infty} \cdots \int_{1}^{\infty} f_n(b_1(n),...,b_h(n), x_{h+1}, ..., x_k) \dd F(x_{h+1}) \cdots \dd F(x_k).
\end{equation}

We now compute $S_{\mathcal{B}}^{(h)}(n)$ in three different cases for $h$.

\begin{lemma}[Less than $k-2$ hubs]\label{lem:contribution_hubs<k-2}
Let $1 \leq h \leq k-3$ and $\mathcal{B} = \{b_i(n)\}_{i=1,...,h}$. Then
\[ S_{\mathcal{B}}^{(h)}(n) = O(n^{-\alpha(k-h)/2}). \]
\end{lemma}

\begin{proof}
We have
\begin{align}\label{eq:few_hubs_upper_bound}
S_{\mathcal{B}}^{(h)} & = \int_{1}^{\infty} \cdots \int_{1}^{\infty} f_n(b_1(n),...,b_{h}(n), x_{h+1},..., x_k) \dd F(x_{h+1}) \cdots \dd F(x_k)\nonumber\\
 & \leq \int_{1}^{\infty} \cdots \int_{1}^{\infty} \prod_{h+1 \leq i < j \leq k} \left(1 \wedge \frac{x_i x_j}{\mu n}\right) \dd F(x_{h+1}) \cdots \dd F(x_k).
\end{align}
Furthermore, the right-hand side in the bound of~\eqref{eq:few_hubs_upper_bound} is the probability that a clique of size $k-h$ is formed on $k-h$ uniformly chosen vertices. By Lemma \ref{lem:averagecliquenumber}, this quantity is $\Theta(n^{-\alpha(k-h)/2})$.
\end{proof}

\begin{lemma}[$k-1$ hubs]\label{lem:contribution_hubs>k-2}
Let $\mathcal{B} = \{b_i(n)\}_{i=1,...,k-1}$. Then
\[ S_{\mathcal{B}}^{(k-1)} =  O\left(\frac{\max_i b_i(n)}{n}\right)^{\alpha}. \]
\end{lemma}

\begin{proof}
Let $\bar{b}= \max_i b_i(n)$. Then
\begin{align}
S_{\mathcal{B}}^{(k-1)} &= \int_{1}^{\infty} f_n(b_1(n),...,b_{k-1}(n), x) \dd F(x) \nonumber\\
&= \int_{1}^{\infty} \prod_{i=1}^{k-1}\left(1 \wedge \frac{x b_i(n)}{\mu n}\right) \dd F(x) \leq \int_{1}^{\infty}\left(1 \wedge \frac{x \bar{b}}{\mu n}\right)^{k-1} \dd F(x) .
\end{align}
To compute the last integral, we split the integration domain
\begin{align}
S_{\mathcal{B}}^{(k-1)} &\leq \left( \frac{\bar{b}}{\mu n} \right)^{k-1} \int_{1}^{\mu n/\bar{b}} x^{k-1} \dd F(x) + \int_{\mu n/\bar{b}}^{\infty} \dd F(x) \nonumber\\
&= \left( \frac{\bar{b}}{\mu n} \right)^{k-1} \frac{(\mu n / \bar{b})^{k-\alpha-1} - 1}{k-\alpha-1}(1+o(1)) + \bar{F}(\mu n / \bar{b}) \nonumber\\
&= \Theta\left(\frac{\bar{b}}{n}\right)^{\alpha}.
\end{align}
\end{proof}

\begin{lemma}[$k-2$ hubs]\label{lem:contribution_hubs=k-2}
Assume that $\mathcal{B} = \{b_i(n)\}_{i=1,...,k-2}$ are ordered increasingly and consider $S_{\mathcal{B}}^{(k-2)}$. When $b(n) = b_1(n) = ... = b_{k-2}(n)$, we have
\begin{equation}\label{eq:boundk-2equal}
    S_{\mathcal{B}}^{(k-2)}(n) =\frac{K_1}{n}\left(\frac{b(n)}{n}\right)^{2(\alpha-1)}(1+o(1))
\end{equation}
for some $K_1 > 0$. In particular, when $b(n)$ is regularly varying of index $\alpha_b$, $S_{\mathcal{B}}^{(k-2)}$ is regularly varying of index $2(\alpha-1)(\alpha_b-1)-1$. \\ Moreover, if $b_1(n) \ll b_{k-2}(n)$, then
\begin{equation}\label{eq:boundk-2different}
    S_{\mathcal{B}}^{(k-2)}(n) = o\left(\frac{1}{n}\left(\frac{b_{k-2}(n)}{n}\right)^{2(\alpha-1)}\right)
\end{equation}
\end{lemma}

\begin{proof}
In this setting
\begin{equation}
    S_{\mathcal{B}}^{(k-2)}(n) = \int_{0}^{\infty}\int_{0}^{\infty} f_n(b_1(n),...,b_{k-2}(n),x,y) \dd F(x) \dd F(y).
\end{equation}
Using the definition of $f_n$ gives
\begin{equation}
    S_{\mathcal{B}}^{(k-2)}(n) = \int_{0}^{\infty}\int_{0}^{\infty} \left(\frac{xy}{\mu n} \wedge 1\right) \prod_{i=1}^{k-2} \left(\frac{x b_i(n)}{\mu n} \wedge 1 \right) \left(\frac{y b_i(n)}{\mu n} \wedge 1 \right) \dd F(x) \dd F(y).
\end{equation}
In the region $xy>\mu n$, the integrand is upper bounded by 1, so that the contribution of that region to $S_{\mathcal{B}}^{(k-2)}(n)$ can be upper bounded by $ \Prob{xy>\mu n}$, which is $\Theta(n^{-\alpha}\log(n))$. In particular, from the Lemma assumption, $b_i(n)$ is slowly varying with index in $(1/2,1)$ for any $i$, therefore $n^{-\alpha}\log(n)$ is negligible compared to the right hand sides of Equations \eqref{eq:boundk-2equal}, \eqref{eq:boundk-2different}. Indeed, when $b_1(n)=...=b_{k-2}(n)=b(n)$, 
$$\frac{1}{n}\left(\frac{b(n)}{n}\right)^{2(\alpha-1)} > n^{-1}(n^{\delta-1/2})^{2(\alpha -1)} = \omega(n^{-\alpha + 2\delta(\alpha-1)}) = \omega(n^{-\alpha}\log(n)),$$ 
for some $\delta > 0$, and similarly $$\frac{1}{n}\left(\frac{b_{k-2}(n)}{n}\right)^{2(\alpha-1)} = \omega(n^{-\alpha}\log(n)).$$

Next, assume that $xy < \mu n$. First, observe that the integral in this region is asymptotically lower bounded by the integral restricted to $x< \sqrt{\mu n}, y < \sqrt{\mu n}$, and it is upper bounded asymptotically by the integral where the constraint $xy<\mu n$ is dropped. These two quantities are asymptotically identical. Therefore, we only need to solve the equivalent separable and symmetric integral:
\begin{equation}
    S_{\mathcal{B}}^{(k-2)}(n) = \frac{1}{\mu n} \left[\int_{1}^{\infty} x \prod_{i=1}^{k-2} \left(\frac{x b_i(n)}{\mu n} \wedge 1 \right) \dd F(x) \right]^2(1+o(1)).
\end{equation}
For each $i$, the minimum in the product is achieved at 1 when $x > \frac{\mu n}{b_i(n)}$. Hence, we can solve the integral by splitting the integration domain into sub-intervals $I_\ell := \left[\mu n/b_{\ell+1}(n),\mu n/b_\ell(n)\right]$
with $\ell=0,...,k-2$ (and conventionally $I_{k-2} = [0,\mu n/b_{k-2}(n)], I_0 = [\mu n /b_1(n), \infty)$). 
Then,
\begin{align}\label{eq:boundk-2explicit}
    S_{\mathcal{B}}^{(k-2)}(n) &=\frac{1}{\mu n} \left(\int_{\mu n/b_1(n)}^{\infty} x \dd F(x) + \sum_{\ell=1}^{k-2} \frac{\prod_{j=1}^\ell b_j(n)}{(\mu n)^\ell} \int_{I_\ell} x^{\ell+1} \dd F(x) \right)^2(1+o(1)) \nonumber \\
    &= (1+o(1))\frac{1}{\mu n} \Bigg(\frac{\alpha}{\alpha-1}\left(\frac{b_1(n)}{\mu n}\right)^{\alpha-1} + \\
    & \hspace{1cm} \sum_{\ell=1}^{k-2}\frac{\alpha}{\ell - \alpha + 1}\frac{\prod_{j=1}^\ell b_j(n)}{(\mu n)^{\alpha-1}} \left(b_\ell(n)^{-(\ell+1-\alpha)} - b_{\ell+1}(n)^{-(\ell+1-\alpha)}\right) \Bigg) ^2 
\end{align}
When $b_1(n) = b_2(n) = ... = b_{k-2}(n)$ most of the terms cancel out and
\begin{equation}
    S_{\mathcal{B}}^{(k-2)}(n) = \frac{1}{\mu n}\left(\frac{\alpha(k-2)}{(\alpha-1)(k-\alpha-1)} \left(\frac{b_i(n)}{\mu n}\right)^{\alpha -1}\right)^2(1+o(1))
\end{equation}
which proves \eqref{eq:boundk-2equal}. 

Next, assume that $b_1(n) \ll b_{k-2}(n)$. On the right-hand side of \eqref{eq:boundk-2explicit} the term
\[ \left(\frac{b_1(n)}{\mu n}\right)^{\alpha-1} = o\left(\frac{b_{k-2}(n)}{n}\right)^{\alpha -1} \]
and for each term in the sum when $\ell > 1$ (using the bound $b_2(n) \leq b_3(n) \leq ... \leq b_{k-2}(n)$)
\begin{align}
\frac{\prod_{j=1}^{\ell} b_j(n)}{(\mu n)^{\alpha-1}}\left(b_\ell(n)^{-(\ell+1-\alpha)} - b_{\ell+1}(n)^{-(\ell+1-\alpha)}\right) & \leq  \frac{b_1(n) b_{k-2}(n)^{\ell -1}}{(\mu n)^{\alpha -1}}b_{k-2}(n)^{-(\ell +1 -\alpha)}\nonumber\\
& = \frac{b_1(n)}{b_{k-2}(n)}\left( \frac{b_{k-2}(n)}{\mu n} \right)^{\alpha-1} \nonumber\\
& = o\left(\frac{b_{k-2}(n)}{n}\right)^{\alpha -1} \end{align}
Hence~\eqref{eq:boundk-2different} follows.
\end{proof}

\subsection{LD for conditional expected cliques}
Denote by $L_n(z)$ the number of vertices with a weight larger than $z$. We now show that whenever the vertex weights are all smaller than $\varepsilon c_a(n)$, $\mathcal{K}_k \leq (1+a) m_n^{(k)}$  with high probability.

\begin{proposition}[No hubs]\label{prop:no_hubs}
Assume that $\alpha > 2 - 2/k$ and let $c_a(n)$ be defined as in \eqref{eq:smallest_hub_size}. Then, there exists $\varepsilon > 0$ such that 
\begin{equation}\label{eq:deviation0hubs}
    \Prob{C_n > (1 + a) m_n^{(k)}; L_n(\varepsilon c_a(n)) = 0} = o(n^{-\beta})
\end{equation}
for arbitrary $\beta$.
\end{proposition}

\begin{proof}
We begin by decomposing 
\begin{align}
    &\Prob{C_n > (1+a)m_n^{(k)} \,;\, L_n(\varepsilon c_a(n)) = 0} \nonumber\\ & \qquad \le \Prob{C_n > (1+a)m_n^{(k)} \,;\, L_n(\varepsilon c_a(n)) = 0 \,;\, E_n(A,\delta,c)} + \Prob{E_n(A,\delta,c)^c}. 
\end{align} 
Since $\prob(E_n(A,\delta,c)^c) = o(n^{-\beta})$ whenever $\lceil c \rceil > \beta/(\alpha\delta)$, it suffices to prove that 
\begin{equation}\label{eq:goal_zero} 
\Prob{C_n > (1+a)m_n^{(k)} \,;\, L_n(\varepsilon c_a(n)) = 0 \,;\, E_n(A,\delta,c)} = 0 \end{equation} for sufficiently small $A,\varepsilon>0$. 

 On the event \[ \{L_n(\varepsilon c_a(n)) = 0\} \cap E_n(A,\delta,c), \] the empirical tail satisfies \[ \bar F_n(x) \le \bar F_n^*(x)\,\mathbbm{1}_{\{x < \varepsilon c_a(n)\}}, \] where $\bar F_n^*$ is defined in~\eqref{eq:empdistrbound}. Since $f_n$ is coordinate-wise non-decreasing, the conditional number of $k$-cliques is bounded by \[ C_n \le \binom{n}{k} \int_1^\infty \!\!\cdots\! \int_1^\infty f_n(x_1,\dots,x_k)\, dF_n^*(x_1)\cdots dF_n^*(x_k). \]
 
 For each coordinate $x_i$, split the domain into:
 \[ \text{small } (s): x_i < a(n), \qquad \text{medium } (m): a(n) \le x_i < b(n), \qquad \text{large } (l): b(n) \le x_i < \varepsilon c_a(n). \] 
 Bounding medium coordinates by $b(n)$ and large coordinates by $\varepsilon c_a(n)$, we obtain 
 \begin{multline}\label{eq:splitCn_rewritten} \frac{C_n}{(1+A)^k} \le \sum_{\substack{s,m,l\ge 0 \\ s+m+l = k}} \binom{n}{k}\, \bar F(a(n))^m \left(\frac{c}{n}\right)^l \\ \times \int_0^\infty\!\!\cdots\!\int_0^\infty f_n{ \underbrace{x_1,\dots,x_s}_{s\text{ small}}, \underbrace{b(n),\dots,b(n)}_{m\text{ medium}}, \underbrace{\varepsilon c_a(n),\dots,\varepsilon c_a(n)}_{l\text{ large}} } dF(x_1)\cdots dF(x_s). 
 \end{multline} 
 We now estimate each contribution according to $s$.
 
 \noindent\textbf{Case $s = k$ (all small).} This term is at most $m_n^{(k)}$.
 
 \noindent\textbf{Case $2 < s < k$.} Since $\bar F(a(n)) = O(n^{-1})$, we have \[ \bar F(a(n))^m (c/n)^l = O{\bar F(a(n))^{m+l}}. \] The integral is bounded by $S_{\mathcal B}^{(k-s)}(n)$, so by Lemma~\ref{lem:contribution_hubs<k-2}, \[ \binom{n}{k} \bar F(a(n))^{m+l} S_{\mathcal B}^{(k-s)}(n) = O\!\left(n^k n^{-(m+l)(1-\delta\alpha)} n^{-s\alpha/2}\right) = O\!\left(n^{s(2-\alpha)/2 + \delta\alpha(m+l)}\right). \] Since $s<k$ and $\delta$ can be chosen arbitrarily small, these terms are $o(m_n^{(k)})$.
 
 \noindent\textbf{Case $s = 1$.} Again $\bar F(a(n))^m (c/n)^l = O(\bar F(a(n))^{m+l})$. Lemma~\ref{lem:contribution_hubs>k-2} yields 
 \[ \binom{n}{k} \bar F(a(n))^{k-1} S_{\mathcal B}^{(k-1)}(n) = O\!\left(n^k n^{-(k-1)(1-\delta\alpha)} n^{-\alpha(1-\alpha^*)}\right) = O\!\left(n^{\alpha\delta(k-1) + 1 - \alpha(1-\alpha^*)}\right). 
 \] 
 Since $\alpha\delta(k-1)$ can be made arbitrarily small and 
 \[ 1 - \alpha(1-\alpha^*) \le \frac{k(\alpha-2)}{2} \] 
 for all $k\ge 3$ and $\alpha\in(2-2/k,2)$, these terms are also $o(m_n^{(k)})$. 
 
 \noindent\textbf{Case $s = 2$.} If $m>0$, Lemma~\ref{lem:contribution_hubs=k-2} with $b_1(n)=b(n)$ and $b_{k-1}(n)=\varepsilon c_a(n)$ gives \[ \binom{n}{k} \bar F(a(n))^m (c/n)^l S_{\mathcal B}^{(k-2)}(n) = o\!\left(n \left(\frac{c_a(n)}{n}\right)^{2(\alpha-1)}\right) = o(m_n^{(k)}), \] using Lemma~\ref{lem:ca(n)}. If $l = k-2$, Lemma~\ref{lem:contribution_hubs=k-2} with all large coordinates equal to $\varepsilon c_a(n)$ yields \[ \binom{n}{k} (c/n)^l S_{\mathcal B}^{(k-2)}(n) = c^{k-2}\varepsilon^{2(\alpha-1)} a\, m_n^{(k)}(1+o(1)). \]

 Summing all contributions in~\eqref{eq:splitCn_rewritten}, we obtain, on $ E_n(A,\delta,c) \cap \{L_n(\varepsilon c_a(n)) = 0\}$
 , that
 \[ C_n \le (1+A)^k {1 + c^{k-2}\varepsilon^{2(\alpha-1)} a} m_n^{(k)}. 
 \] 
 For sufficiently small $A,\varepsilon>0$, the right-hand side is strictly less than $(1+a)m_n^{(k)}$, which contradicts the event in~\eqref{eq:goal_zero}. Thus, the probability in~\eqref{eq:goal_zero} is zero, completing the proof.
\end{proof}

Next, we investigate the probability of exceeding $(1+a)m_n^{(k)}$ $k$-cliques when fewer than $k-2$ large hubs of size at least $c_a(n)$ are present.

\begin{proposition}[Fewer than $k-2$ hubs]\label{prop:not_enough_hubs}
Assume that $\alpha > 2 - 2/k$ and let $c_a(n)$ be defined as in \eqref{eq:smallest_hub_size}. Then, for any $1 \leq i < k-2$, there exists $\varepsilon > 0$ such that 
\begin{equation}
    \Prob{C_n > (1+a) m_n^{(k)} ; L_n(\varepsilon c_a(n)) = i} = O(n \Prob{W > c_a(n)})^{k-2}.
\end{equation}
\end{proposition}
\begin{proof}
Since $L_n(\varepsilon c_a(n))$ is binomial with parameters $n$ and $p_n := \Prob{W > \varepsilon c_a(n)}$ and $p_n = o(1/n)$, we have, for each fixed $i \in \mathbb{N}$, 
\[ \Prob{L_n(\varepsilon c_a(n)) = i} =  (1+o(1)) {( n \Prob{W > \varepsilon c_a(n)} })^i. \] 
Hence it suffices to prove that, for every fixed $i$ with $1 \le i < k-2$, \begin{equation}\label{eq:induction_step} \Prob{ C_n > (1+a) m_n^{(k)} \,\big|\, L_n(\varepsilon c_a(n)) = i } = O{( {( n \Prob{W > c_a(n)} })^{k-i-2} }). \end{equation}
Condition on the event $\{L_n(\varepsilon c_a(n)) = i\}$. For any $k$-clique that contains a certain number $j \in \{0,\dots,i\}$ of these hubs, the remaining $k-j$ vertices must form a $(k-j)$-clique among the non-hub vertices. Let $C_n^{(k-j)}$ denote the total number of $(k-j)$-cliques in the graph. Then, on $\{L_n(\varepsilon c_a(n)) = i\}$, the total number of $k$-cliques satisfies
\[ C_n \le \sum_{j=0}^{i} C_n^{(k-j)}. \] Fix $\bar a := a/(i+1)$. If $C_n > (1+a)m_n^{(k)}$, then at least one of the $i+1$ terms in the above sum must exceed its mean by at least $\bar a m_n^{(k)}$. More precisely, 
\begin{align}\label{eq:uboundkjcliques} &\Prob{ C_n > (1+a) m_n^{(k)} \,\big|\, L_n(\varepsilon c_a(n)) = i} \\ &\qquad \le \Prob{ C_n > (1+a) m_n^{(k)} \,\big|\, L_n(\varepsilon c_a(n)) = 0} \nonumber \\ &\qquad\qquad + \sum_{j=1}^i \Prob{ C_n^{(k-j)} > \bar a m_n^{(k)} }.\nonumber 
\end{align}
By Proposition~\ref{prop:no_hubs}, we can choose $\varepsilon>0$ such that
\[ \Prob{ C_n > (1+\bar a) m_n^{(k)} \,\big|\, L_n(\varepsilon c_a(n)) = 0} \] decays faster than any fixed polynomial in $n$. In particular, this term is $O{( (n \Prob{W > c_a(n)})^{k-i-2} })$, since $n \Prob{W > c_a(n)} \to 0$ at a polynomial rate under our assumptions.

We now control the remaining terms in~\eqref{eq:uboundkjcliques} by induction on the clique size. From~\cite{stegehuis2023}, the statement of Proposition~\ref{thm:LDcliques} is known for $k=3$ (triangles), which serves as the base case. Fix $k\ge 4$ and assume that Proposition~\ref{thm:LDcliques} holds for all clique sizes $3,\dots,k-1$. Then, for any $j \in \{1,\dots,k-2\}$ and any $\tilde a>0$, 
\begin{equation}\label{eq:induction_assumption} \Prob{ C_n^{(k-j)} > (1+\tilde a) m_n^{(k-j)} } = O{( {( n \Prob{W > c^{(k-j)}_{\tilde a}(n)} })^{k-j-2} }), \end{equation} 
where $m_n^{(k-j)}$ is the expected number of $(k-j)$-cliques, and $c^{(k-j)}_{\tilde a}(n)$ is the threshold (as in~\eqref{eq:smallest_hub_size}) corresponding to having an expected excess of $\tilde a m_n^{(k-j)}$ $(k-j)$-cliques. Fix $j \in \{1,\dots,i\}$. We choose $\tilde a = \tilde a(j,n)$ such that 
\begin{equation}\label{eq:tilde_a_choice} (1+\tilde a) m_n^{(k-j)} = \bar a m_n^{(k)}. \end{equation} 
By Lemma~\ref{lem:averagecliquenumber}, we have $m_n^{(k-j)} = o(m_n^{(k)})$, so~\eqref{eq:tilde_a_choice} implies \[ \tilde a = \frac{\bar a m_n^{(k)}}{m_n^{(k-j)}} - 1 = (1+o(1)) \frac{\bar a m_n^{(k)}}{m_n^{(k-j)}} = \Theta{( n^{j(1-\alpha/2)} }). \] Applying Lemma~\ref{lem:ca(n)} to $(k-j)$-cliques, we obtain \[ c^{(k-j)}_{\tilde a}(n) = \Theta\!\left( \tilde a^{\frac{1}{2(\alpha-1)}} n^{1 - \frac{2 - (k-j)(2-\alpha)}{4(\alpha-1)}} \right) = \Theta\!\left( n^{1 - \frac{2-k(2-\alpha)}{4(\alpha-1)}}\right) = \Theta{( c_a(n) }). \] Hence $\Prob{W > c^{(k-j)}_{\tilde a}(n)}$ is of the same order as $\Prob{W > c_a(n)}$, and by~\eqref{eq:induction_assumption}, 
\begin{align*} \Prob{ C_n^{(k-j)} > \bar a m_n^{(k)} } &= \Prob{ C_n^{(k-j)} > (1+\tilde a) m_n^{(k-j)} } \\ &= O{( {( n \Prob{W > c^{(k-j)}_{\tilde a}(n)} })^{k-j-2} }) \\ &= O{( {( n \Prob{W > c_a(n)} })^{k-j-2} }). \end{align*} 
Since $j \le i$ and $n \Prob{W > c_a(n)} \to 0$, we have \[ {( n \Prob{W > c_a(n)} })^{k-j-2} \le {( n \Prob{W > c_a(n)} })^{k-i-2}, \] so, for each $j=1,\dots,i$, \[ \Prob{ C_n^{(k-j)} > \bar a m_n^{(k)} } = O{( {( n \Prob{W > c_a(n)} })^{k-i-2} }). \] Summing over $j=1,\dots,i$ and combining with~\eqref{eq:uboundkjcliques}, we obtain \[ \Prob{ C_n > (1+a) m_n^{(k)} \,\big|\, L_n(\varepsilon c_a(n)) = i} = O{( {( n \Prob{W > c_a(n)} })^{k-i-2} }), \] which is exactly~\eqref{eq:induction_step}. This completes the induction and the proof of the proposition.
\end{proof}

\begin{proof}[Proof of Proposition \ref{prop:LDconditionalcliques}]

We decompose \[ \Prob{C_n > (1+a)m_n^{(k)}} = \sum_{i=0}^{n} \Prob{ C_n > (1+a)m_n^{(k)} \,;\, L_n(\varepsilon c_a(n)) = i }, \] and estimate each term separately. 

\smallskip \noindent\textbf{Case $i=0$.} By Proposition~\ref{prop:no_hubs}, \[ \Prob{ C_n > (1+a)m_n^{(k)} \,;\, L_n(\varepsilon c_a(n)) = 0 } \] decays faster than any fixed polynomial in $n$. 

\smallskip \noindent\textbf{Case $1 \le i < k-2$.} Proposition~\ref{prop:not_enough_hubs} gives 
\[ \Prob{ C_n > (1+a)m_n^{(k)} \,;\, L_n(\varepsilon c_a(n)) = i } = O\!\left( {( n \Prob{W > c_a(n)} })^{k-2} \right). \] 
\smallskip \noindent\textbf{Case $i = k-2$.} We use the trivial bounds \begin{align}\label{eq:bounds_k-2} &\Prob{ C_n > (1+a)m_n^{(k)} \,;\, L_n(c_a(n)) = k-2 } \nonumber \\ & \quad \le \Prob{ C_n > (1+a)m_n^{(k)} \,;\, L_n(\varepsilon c_a(n)) = k-2 } \\ & \qquad \le \Prob{ L_n(\varepsilon c_a(n)) = k-2 }. \nonumber \end{align} 
For the lower bound in~\eqref{eq:bounds_k-2}, Lemma~\ref{lem:contribution_hubs=k-2} applied with $b(n)=c_a(n)$ shows that the event 
$\{L_n(c_a(n)) = k-2\} $
already produces at least $(1+a)m_n^{(k)}$ cliques with probability asymptotically of order \[ (n \Prob{W > c_a(n)})^{k-2}. \] 
Thus, the left-hand side of~\eqref{eq:bounds_k-2} is asymptotically bounded from below by a constant multiple of$ (n \Prob{W > c_a(n)})^{k-2}. $ 
For the upper bound, note that $L_n(\varepsilon c_a(n))$ is binomial with parameters $n$ and \[ p_n := \prob(W > \varepsilon c_a(n)). \] $\Prob{W > c_a(n)}$ is regularly varying and \[ \prob(W > \varepsilon c_a(n)) = (1+o(1)) \prob(W > c_a(n)), \] because $\varepsilon>0$ is fixed. Therefore, 
    \[ \Prob{L_n(\varepsilon c_a(n)) = k-2}  = \Theta\!\left( (n \Prob{W > c_a(n)})^{k-2} \right). \]  Combining the upper and lower bounds in~\eqref{eq:bounds_k-2}, we conclude that \[ \Prob{ C_n > (1+a)m_n^{(k)} \,;\, L_n(\varepsilon c_a(n)) = k-2 } = \Theta\!\left( (n \Prob{W > c_a(n)})^{k-2} \right). \] 
\smallskip \noindent\textbf{Case $i > k-2$.} Since $L_n(\varepsilon c_a(n))$ is binomial with mean $n\prob(W > \varepsilon c_a(n)) = o(1)$, \[ \Prob{ L_n(\varepsilon c_a(n)) \ge k-3 } = O\!\left( (n \Prob{W > c_a(n)})^{k-3} \right). \] Thus, \[ \Prob{( C_n > (1+a)m_n^{(k)} \,;\, L_n(\varepsilon c_a(n)) > k-2 }) \le O\!\left( (n \Prob{W > c_a(n)})^{k-3} \right). \] 
\smallskip Summing all contributions, the dominant term arises from $i = k-2$, and the proposition follows. 
\end{proof}

\section{LD for number of cliques (proof of Theorem~\ref{thm:LDcliques})}\label{sec:proofcliques}
We now show that the probability that the number of cliques deviates from its expectation conditionally on the vertex weights decays exponentially fast. This relies on the fact that any $k$ clique contains $k-2$ triangles sharing an edge. To do so, we first show that the probability that a single edge between two vertices with product of their weights $W_iW_j<\mu n$, appears in many $k$-cliques is small:
\begin{lemma}\label{lem:kcliqueshare}
Suppose that $1<\alpha<2$. Then, when $L>n^{(2-\alpha)(k-2)/2+\varepsilon}$ for some $\varepsilon>0$,
    \begin{equation}
        \Prob{\{i,j\}\text{ appears in $\geq L$ $k$-cliques}\mid W_i W_j<\mu n}\leq \exp(-c_1 L^{1/(k-2)}),
    \end{equation}
    for some $c_1>0$.
\end{lemma}
\begin{proof}
First of all, when an edge $\{i,j\}$ appears in $L$ $k$-cliques, it has to appear in at least $L^{1/(k-2)}$ triangles. Indeed, any clique containing the edge $\{i,j\}$ contains $k-2$ triangles containing $\{i,j\}$. Therefore, 
    \begin{align}
        &\Prob{\{i,j\}\text{ appears in }\geq L\quad  k-\text{cliques}\mid W_iW_j<\mu n}\nonumber\\
        &\leq \Prob{\{i,j\}\text{ appears in }\geq L^{1/(k-2)}\quad  \text{triangles}\mid W_iW_j<\mu n}\nonumber\\
        & \leq \exp(-c_1 L^{1/(k-2)}),
    \end{align}
    where the last step follows from~\cite[Lemma 4.4]{stegehuis2023}, which only holds for $L>n^{(2-\alpha)(k-2)/2+\varepsilon}$.
\end{proof}

Lemma~\ref{lem:kcliqueshare} only bounds the number of shared edges in cliques for edges with incident vertices of relatively low degree. To bound the probability that the number of $k$-cliques is large, we also need to deal with overlapping cliques containing high-degree vertices. To do so, we define the event $E_w = \{W_1=w_1,...,W_n=w_n\}$ and we set $g_n= \sum_{i_1<i_2<\dots<i_k} f_n(w_1,w_2,\dots,w_k)$ to be the expected number of $k$-cliques conditional on the weights. Let $K_n^k(w)$ be the random variable describing the number of $k$-cliques conditionally on $E_w$.
Finally,we  set for $\zeta>0$,
\begin{equation}
    \label{eq-def-J}
    J(\eta) = (1+\zeta) \log (\zeta + 1/(1+\zeta)) /3.
\end{equation}
We now show that the probability that $K^k_n(w)$ is larger than $g_n$ is small:
\begin{lemma}
\label{lem-conditionedchatterjeetriangles} 
There exists an $\varepsilon>0$ such that, for $\zeta>0$, and all $w=(w_1,...,w_n)$:
\begin{equation}
     \Prob{K^k_n(w) >(1+\zeta )g_n} \leq e^{-J(\zeta) g_n/n^{(2-\alpha)(k-2)/2+\delta}}+ e^{-n^\varepsilon}.
\end{equation}
  \end{lemma}
  
\begin{proof}[Proof of Lemma \ref{lem-conditionedchatterjeetriangles}]
By Lemma~\ref{lem:kcliqueshare}, when $w_iw_j<\mu n$, and choosing $L=n^{(2-\alpha)(k-2)/2+\delta}$, 

\begin{equation}\label{eq:ijtriangles}
	\Prob{ \{i,j\} \text{ in }\geq n^{(2-\alpha)(k-2)/2+\delta}\quad k \text{-cliques}}\leq \exp(-K_1 n^{(2-\alpha)/2+\delta/(k-2)}),
\end{equation} 
for some $K_1>0.$
This indicates that with probability at least $1-n^2\exp(-K_1 n^{(2-\alpha)/2+\delta/(k-2)}) \geq 1- \exp(-n^\varepsilon)$, all edges between vertices of weights $w_iw_j<\mu n$ appear in at most $n^{(2-\alpha)(k-2)}$ $k$-cliques for some $\varepsilon>0$.
We now work on this event, which we call $\mathcal{E}'$.

We set $X_{i_1,\dots,i_k}=X'_{i_1,\dots,i_k}$ as the indicator that a $k$-clique is present on vertices $i_1,\dots,i_k$. Furthermore, we set $X_{i_1,\dots,i_k(j_1,\dots,j_k)}=X_{i_1,\dots,i_k}$ when $|\{i_1,\dots,i_k,j_1,\dots,j_k\}|\geq 2k-2$. This corresponds to the setting in which two cliques overlap at a vertex or do not overlap at all. When $|\{i_1,\dots,i_k,j_1,\dots,j_k\}|< 2k-2$, we set $X_{i_1,\dots,i_k(j_1,\dots,j_k)}=X_{i_1,\dots,i_k}$ when the overlap of $i_1,\dots,i_k$ and $j_1,\dots,j_k$ occurs only at edges with $w_uw_v>\mu n$, and we set $X_{i_1,\dots,i_k(j_1,\dots,j_k)}= 0$ otherwise. Then, $\sum_{i_1,\dots,i_k}X_{i_1,\dots,i_k(j_1,\dots,j_k)}$ and $X_{j_1,\dots,j_k}$ are independent. Indeed, when two $k$-cliques do not overlap, or only overlap at a single vertex, their presence is independent, conditionally on the weights, as the edge indicators are independent conditionally on the weights. When the edge overlap(s) occur only at edges that are present with probability one, the presence of the two $k$-cliques is also independent conditionally on the weights. In all other cases, $X_{i_1,\dots,i_k(j_1,\dots,j_k)}=0$, which is also independent of $X_{j_1,\dots,j_k}$. 

We now bound the number of possible $X_{i_1,\dots,i_k(j_1,\dots,j_k)}$ that are manually put to zero.
There are $\binom{k}{l}$ possible sets of $l$ vertices of the $k$-cliques that can be part of an $l$-overlap with $j_1,\dots,j_k$. As there is at least one edge with $w_uw_v<\mu n$ when $X_{i_1,\dots,i_k(j_1,\dots,j_k)}$ that is manually put to zero, there can be at most $n^{(3-\tau)/2+\delta/(k-2)}$ cliques that overlap with the $l$-set of vertices.

Summing over $l$, on the event $\mathcal{E}'$,
\begin{align}
	k! K_n^k(w) &=\sum_{i_1,\dots,i_k}X_{i_1,\dots,i_k} \nonumber \\ &\leq \sum_{i_1,\dots,i_k}X_{i_1,\dots,i_k(j_1,\dots,j_k)}+\Big( {\binom{k}{2}}+ {\binom{k}{3}}+\dots + {\binom{k}{k}}\Big)n^{(2-\alpha)(k-2)/2+\delta}.
\end{align}
 Then, by~\cite[Lemma Theorem 3.1]{chatterjee2012} with $a=\Big( {\binom{k}{2}}+ {\binom{k}{3}}+\dots + {\binom{k}{k}}\Big)n^{(2-\alpha)(k-2)/2+\delta}$, $t=(1+\zeta)g_n$ and $\lambda=g_n$ finishes the proof. 
\end{proof}


\begin{proof}[Proof of Theorem \ref{thm:LDcliques}]
We begin with an asymptotic upper bound. Using Lemma \ref{lem-conditionedchatterjeetriangles}, for $\zeta > 0$
\begin{align}
    \Prob{\mathcal{K}_n^{(k)} > (1 + a)m_n^{(k)}} &\leq \Prob{\mathcal{K}_n^{(k)} \geq (1 + \zeta)C_n} +  \Prob{C_n/(1+\zeta) \geq \mathcal{K}_n^{(k)} \geq (1 + a)m_n^{(k)}} \nonumber\\
    &\leq e^{-n^{\varepsilon}} + \Prob{C_n \geq (1 + \zeta)(1 + a)m_n^{(k)}}\nonumber\\
    & \leq H\left(n \Prob{W > c_a(n)}\right)^{k-2},
\end{align}
for some constant $H$, where the last line follows from Proposition~\ref{prop:LDconditionalcliques}.

To obtain an asymptotic lower bound, we condition on the event that $k-2$ vertices have sufficiently large weights. In particular, let $\mathcal{K}_n^{(k)}(\delta, a)$ be the number of $k$-cliques in IRG, such that: 
\begin{itemize}
    \item either all nodes have weights at most $n^{1/2+\delta}$,
    \item or, two nodes have weights at most $n^{1/2+\delta}$ and the rest have weights at least $c_a(n)$.
\end{itemize}
Then,
\begin{equation}
    \Prob{\mathcal{K}_n^{(k)} \geq (1+a)m_n^{(k)}} \geq (n \Prob{W > c_a(n)})^{k-2} \Prob{\mathcal{K}_n^{(k)}(\delta,a) \geq (1+a)m_n^{(k)}}.
\end{equation}
From here, it is sufficient to show that $\mathcal{K}_n^{(k)}(\delta,a) / m_n^{(k)}$ converges in probability to $(1+a)$. To this end, we employ a second-moment method. First, observe that 
\begin{align}
    \E{\mathcal{K}_n^{(k)}(\delta,a)} &= \binom{n-(k-2)}{k} \int_{0}^{n^{1/2 + \delta}} \cdots \int_{0}^{n^{1/2 + \delta}} f_n(x_1,...,x_k) \dd F(x_1) \cdots \dd F(x_k) \label{eq:Kda-1} \\
    &\hspace{-1cm}+ \frac{1}{2}(n-(k-2))^2 \int_{0}^{n^{1/2 + \delta}} \int_{0}^{n^{1/2 + \delta}} f_n(x_1,x_2, c_a(n),...,c_a(n)) \dd F(x_1) \dd F(x_2) \label{eq:Kda-a}
\end{align}
From the proof of Lemma \ref{lem:averagecliquenumber}, the right-hand side term in \eqref{eq:Kda-1} asymptotically equals $m_n^{(k)}$. Instead, from Lemma \ref{lem:contribution_hubs=k-2}, it follows by taking $b(n) = c_a(n)$ that the term in \eqref{eq:Kda-a} is asymptotically equal to $am_n^{(k)}$. Summing up, $\E{\mathcal{K}_n^{(k)}(\delta,a)} = (1+o(1))(1+a) m_n^{(k)}$. Second, we investigate the concentration properties of $\mathcal{K}_n^{(k)}(\delta,a)$. We can write 
\begin{align}
    \mathcal{K}_n^{(k)}(\delta,a) &= \binom{n-(k-2)}{k} \int_{0}^{n^{1/2 + \delta}} \cdots \int_{0}^{n^{1/2 + \delta}} f_n(x_1,...,x_k) \dd F_{n}(x_1) \cdots \dd F_n(x_k) \\
    &+ (n-(k-2))^2 \int_{0}^{n^{1/2 + \delta}} \int_{0}^{n^{1/2 + \delta}} f_n(x_1,x_2, c_a(n),...,c_a(n)) \dd F_n(x_1) \dd F_n(x_2)
\end{align}
where we recall that $F_n$ denotes the empirical distribution function of weights. The distribution function $F_n$ concentrates around $F$, that is, for any $\zeta$ and $\delta$ there exists $\varepsilon$ such that (cfr. Proposition 2.2 in \cite{stegehuis2023})
\begin{equation}
    \Prob{\sup_{x < n^{1/2 + \delta}} \left|\frac{\bar{F}_n(x)}{\bar{F}(x)}  - 1 \right| > \zeta} \leq e^{-n^{\varepsilon}}.
\end{equation}
On this event, $\mathcal{K}_n^{(k)}(\delta,a) \leq (1 + \zeta)^k \E{\mathcal{K}_n^{(k)}(\delta,a)} = (1+o(1))(1 + \zeta)^k(1+a)m_n^{(k)}$, otherwise, we trivially bound $\mathcal{K}_n^{(k)}(\delta,a) \leq n^k$. Thus, 
\begin{equation}
    \E{(\mathcal{K}_n^{(k)}(\delta,a))^2} \leq (1+o(1)) ((1+\zeta)^{k}(1+a)m_n^{(k)})^2 + n^{2k} e^{-n^\varepsilon} = O((1+\zeta)^k(1+a)m_n^{(k)})^2
\end{equation}
Then, by Chebyshev's inequality, for any $\delta$
\begin{equation}
    \limsup_{n \to \infty} \Prob{\left| \frac{\mathcal{K}_n^{(k)}(\delta,a)}{\E{\mathcal{K}_n^{(k)}(\delta,a)}} - 1 \right| > \delta} \leq \frac{O(1+\zeta)^{2k} - 1}{\delta^2}
\end{equation}
and the proof is concluded by taking the limit $\delta \to 0$.

\end{proof}

\section{Proofs for general subgraph occurrence (Theorem~\ref{thm:subgraphdeviation})}
\label{sec:proofsubgraphs}
We now investigate the impact of non-typical large hubs on subgraph counts in IRG. The appearance of a hub of size $n^{\beta}$ with $\beta> 1/\alpha$ is a rare event. Indeed, the probability of a vertex weight to be of the order $n^{\beta}$ is proportional to $n^{\alpha\beta}$, hence the expected number of vertices with such weights is $\Theta(n^{1 - \beta\alpha}) \ll 1$. 

By~\eqref{eq:convergencesubgraphs},  $\Prob{N(H) \leq n^{B(H) + \delta}} \to 1$, for any $\delta > 0$. Thus, $N(H)$ concentrates around $n^{B(H)}$. Still, $N(H)$ can be large when extra-large hubs appear in IRG. To formalize this, we investigate the extra contribution to the expected value $\E{N(H)}$ provided by extra-large hubs. Formally, let $\boldsymbol{\beta} = (\beta_1,...,\beta_k) \in [0,1]^k$ be a sequence where some entries might have value larger than $1/\alpha$, corresponding to extra-large hubs. 
\begin{proof}[Proof of Theorem~\ref{thm:subgraphdeviation}]
\textbf{Lower bound.}
Recall that
\begin{align}
R(H):=\max_{\beta} & \sum_{i \in [k]} \min \left(1 - \alpha\beta_{i},0\right) \label{eq:opt_poly_obj2} \\
\text { s.t. } & \sum_{i \in V_{H}} \max \left(1 - \alpha\beta_{i},0\right)+\sum_{\{i, j\} \in E_{H}} \min \left(\beta_{i}+\beta_{j}-1,0\right) \geq \gamma .\label{eq:opt_poly_constr2}
\end{align}
Let $\boldsymbol{\beta}^*$ be the vector that optimizes $R(H)$. We now consider the number of copies of $H$ on vertices with degrees proportional to $\boldsymbol{\beta}^*$, and will show that they cause $n^\gamma$ copies of $H$ with high probability. To be more precise, for fixed $0<\varepsilon<1$, we define
\begin{equation}
    M^{(\boldsymbol{\beta})}(\varepsilon) := \left\{ (v_1,...,v_k) : w_{v_i} \in [1, 1/\varepsilon]n^{\beta_i}, \forall i \in [k] \right\},
\end{equation}
for any given $\boldsymbol{\beta} = (\beta_1,...,\beta_k)$. Let $N(H,M^{(\boldsymbol{\beta})}(\varepsilon))$ be the conditional expectation of the number of subgraphs $H$ appearing on vertex $k$-tuples in $M^{(\boldsymbol{\beta})}(\varepsilon)$, conditionally on the weights. Then,
\begin{align}
    N(H,M^{(\boldsymbol{\beta})}(\varepsilon)) &= \sum_{\bv \in M^{(\boldsymbol{\beta})}(\varepsilon)} \prod_{(i,j) \in E_H}\min\Big(\frac{w_{v_i}w_{v_j}}{n},1\Big).
\end{align}
For $\boldsymbol{\beta}^*$, this yields
\begin{align}\label{eq:MnNnsize}
    N(H,M^{(\boldsymbol{\beta}^*)}(\varepsilon)) &\geq  | M^{(\boldsymbol{\beta}^*)}(\varepsilon)| \prod_{(i,j) \in E_H}n^{\min(\beta_i^*+\beta_j^*-1,0)  }.
\end{align}

 The number of vertices of weight within $[1,1/\varepsilon]n^\beta$ is a Binomial random variable with mean $n n^{\beta\alpha}(1-\varepsilon^{(\alpha)})$. When $\beta\leq 1/\alpha$, then this mean is constant or tends to infinity when $n$ grows, so that the number of vertices of weight within $[1,1/\varepsilon]n^\beta$ is $\Theta_p(n^{1 - \beta\alpha})$.

Let $\mathcal{T} = \{i\in[k]: \beta_i^*>1/\alpha\}$. We then define the event 
\begin{equation}
    \mathcal{E} = \{ \exists v_1,\dots, v_{|\mathcal{T}|}\in [n]: w_{v_i}\geq \varepsilon n^{\beta_i^*} \quad \forall i\in [|\mathcal{T}|] \},
\end{equation}
which implies that there is as least one set of $|\mathcal{T}|$ vertices with extra large degrees.
Conditionally on $\mathcal{E}$, 
\begin{align}
    | M^{(\boldsymbol{\beta}^*)}(\varepsilon)| & \geq \prod_{i:\beta_i^* \leq 1/\alpha}\Theta_p(n^{1 - \alpha\beta_i^*})\nonumber \\
    & =  \Theta_p\Big( n^{\sum_{i\in[k]}\max(1 - \alpha\beta_i^*,0)} \Big).
\end{align}
Combining this with~\eqref{eq:MnNnsize} yields that on $\mathcal{E}$, with high probability, for some $c_2>0$
\begin{align}
    N(H,M^{(\boldsymbol{\beta}^*)}(\varepsilon)) &\geq c_2  n^{\sum_{i\in[k]}\max(1 - \alpha\beta_i^*,0)+\sum_{\{i,j\}\in E_H}\min(\beta_i^*+\beta_j^*-1,0)}\geq c_2 n^{\gamma},
\end{align}
where the last inequality follows from the constraint in the optimization problem for $R(H)$ in~\eqref{eq:opt_poly_constr2}.

We now compute the probability of the event $\mathcal{E}$.
 When $\beta>1/\alpha,$ 
 \begin{equation}
     \Prob{L_n(\varepsilon) n^{\beta}=k}\geq c_\varepsilon (n^{1+\alpha\beta})^k,
 \end{equation}
 for some $c_\varepsilon>0$. Thus, for some $\tilde{c}>0$,
 \begin{equation}
     \Prob{\mathcal{E}} \geq \tilde{c}\prod_{i:\beta_i^*>1/\alpha}n^{1-\beta_i^*\alpha} = \tilde{c}n^{\sum_{i\in[k]}\min(1-\beta_i^*\alpha,0)} = \tilde{c}n^{R(H)}.
 \end{equation}
 Therefore
 \begin{align}
     \lim_{n\to\infty}\frac{\log(\Prob{C^{(n)}(H)>n^\gamma})}{\log(n)} & \geq 
      \lim_{n\to\infty}\frac{\log(\Prob{N(H,M^{(\boldsymbol{\beta}^*)}(\varepsilon))>n^\gamma})}{\log(n)} \nonumber\\
      & \geq \lim_{n\to\infty}\frac{\log(\Prob{\mathcal{E}})}{\log(n)} = R(H).
 \end{align}

\textbf{Upper bound.}
We cover the interval $[1,n]$ with intervals of the type $I_\varepsilon(n^\zeta) = [\varepsilon,1/\varepsilon]n^\zeta$ with $\zeta\in[0,1]$. This is the same as covering $[0,\log(n)]$ with intervals of the type $[\zeta\log(n)+\log(\varepsilon),\zeta\log(n)-\log(\varepsilon)]$, which have length $-2\log(\varepsilon)$. Thus, to cover the interval $[1,n]$, we need $-2\log(n)/\log(\varepsilon)$ such logarithmic intervals. Therefore, $[0,1]^k$ can be covered with a covering $\mathcal{C}$ of size $(-2\log(n)/\log(\varepsilon))^k = \Theta(\log(n)^k)$. Then,
\begin{align}\label{eq:upperbound_split}
    \Prob{C^{(n)}(H)>n^\gamma} & = \Prob{\sum_{\boldsymbol{\beta}\in\mathcal{C}}N(H,M^{(\boldsymbol{\beta})}(\varepsilon))>n^\gamma}\nonumber\\
    & \leq \Prob{\exists \boldsymbol{\beta}\in\mathcal{C}: N(H,M^{(\boldsymbol{\beta})}(\varepsilon))>n^\gamma/\log(n)^k}\nonumber\\
    & \leq \sum_{\boldsymbol{\beta}\in\mathcal{C}}\Prob{N(H,M^{(\boldsymbol{\beta})}(\varepsilon))>n^\gamma/\log(n)^k},
\end{align}
where the last step follows from the union bound.
For $\beta<1/\alpha$, the number of vertices with weight within $[\varepsilon,1/\varepsilon]n^{\beta}$ is a Binomial random variable with diverging mean. Therefore, there exist $C_\varepsilon,c_\varepsilon>0$ such that
\begin{equation}\label{eq:binconcentration}
    \Prob{\sum_{i\in[n]}\mathbbm{1}_{w_i\in [\varepsilon,1/\varepsilon]n^{\beta}}\geq C_\varepsilon n^{1-\beta\alpha}}\leq \exp(-n^{c_\varepsilon}),
\end{equation}
$\forall \boldsymbol{\beta} \in\mathcal{C}$.
Let 
\begin{equation}
	\mathcal{E}_2 = \big\{\sum_{i\in[n]}\mathbbm{1}_{w_i\in [\varepsilon,1/\varepsilon]n^{\beta}}\leq C_\varepsilon n^{1-\beta\alpha} \quad \forall \boldsymbol{\beta} \in\mathcal{C}\big\}.
	\end{equation} 
Then,
\begin{align*}
    &\Prob{C^{(n)}(H)>n^\gamma} \\ 
    & \quad \leq \sum_{\boldsymbol{\beta}\in\mathcal{C}}\Prob{|M^{(\boldsymbol{\beta})}(\varepsilon)|>\tilde{c}n^{\gamma-\sum_{\{i,j\}\in E_H}\min(\beta_i+\beta_j-1,0)}/\log(n)^k \mid \mathcal{E}_2} +\Prob{\mathcal{E}_2}\nonumber\\
    & \quad \leq \sum_{\boldsymbol{\beta}\in\mathcal{C}}\Prob{\prod_{i:\beta_i>1/\alpha}|M^{(\beta_i)}(\varepsilon)|>\tilde{c}n^{\gamma-\sum_{\{i,j\}\in E_H}\min(\beta_i+\beta_j-1,0)-\sum_i(1-\alpha\beta_i)}/\log(n)^k } \nonumber\\
    & \qquad +\exp(-n^c).
\end{align*}
Now, when $\gamma-\sum_{\{i,j\}\in E_H}\min(\beta_i+\beta_j-1,0)-\sum_i(1-\alpha\beta_i)>0$, then this calculates the probability of a polynomially growing number of vertices of degree more than $n^{1/\alpha}$, which happens with exponentially small probability. On the other hand, when $\gamma-\sum_{\{i,j\}\in E_H}\min(\beta_i+\beta_j-1,0)-\sum_i(1-\alpha\beta_i)\leq0$, then only a constant $|M^{(\beta_i)}(\varepsilon)|$ is required, which can be upper bounded by the probability of $M^{(\beta_i)}(\varepsilon)$ being non-empty.

As long as $\gamma-\sum_{\{i,j\}\in E_H}\min(\beta_i+\beta_j-1,0)-\sum_i(1-\alpha\beta_i)\geq0$,
\begin{equation}
    \Prob{M^{(\boldsymbol{\beta})}(\varepsilon) \neq \emptyset}\leq \prod_{i:\beta_i>1/\alpha}(n^{1-\beta_i\alpha})\leq n^{R(H)},
\end{equation}
where the last inequality follows from~\eqref{eq:opt_poly_constr2}. Now the requirement that $N(H,M^{(\boldsymbol{\beta})}(\varepsilon))>n^\gamma/\log(n)^k$ rather than $n^\gamma$ allows for vertices of weight $n^{\beta_i}/\log(n)$ rather than $n^{\beta_i}$, which adjusts the final probability to $n^{R(H)}\log(n)^{k\alpha}$. Combining this with~\eqref{eq:upperbound_split} yields
\begin{equation}
    \Prob{C^{(n)}(H)>n^\gamma}  \leq C n^{-R(H)}\log(n)^{2k\alpha},
\end{equation}
finishing the proof.
\end{proof}

\hspace{5cm}
\paragraph{Funding and competing interests}
Not applicable
\hspace{5cm}

\bibliographystyle{plain}
\bibliography{sn-bibliography}

\end{document}